\documentclass{article}
\usepackage{amsmath,amssymb,amsthm}

\newtheorem{thm}{Theorem}
\newtheorem{lem}{Lemma}
\newtheorem{cor}{Corollary}
\newtheorem{defi}{Definition}
\newtheorem{rem}{Remark}

\newtheorem{rthm}{Theorem}
\newtheorem{rlem}{Lemma}
\newtheorem{rcor}{Corollary}

\usepackage{geometry}
\usepackage{comment}

\newcommand{\sR}{\mathbb{R}}

\newcommand{\cN}{\mathcal{N}}

\title{\textbf{Empirical Process of Multivariate Gaussian under General Dependence}}

\author{Jikai Hou\footnotemark[1] \footnotemark[2]\\
\texttt{houjikai@pku.edu.cn}}
\footnotetext[1]{School of Mathematical Sciences, Peking University, Beijing, China.}
\footnotetext[2]{Part of this work was done while Jikai Hou was visiting University of California, Berkeley.}

\date{October 2019; Revised July 2020}

\begin{document}
\maketitle

\begin{abstract}
This paper explores certain kinds of empirical process with respect to the components of multivariate Gaussian. We put forward some finite sample bounds which hold for multivariate Gaussian under general dependence. We give necessary and sufficient condition for the convergence in probability of the random variable sequence $\displaystyle\left\{\sup_t\vert\widehat{F}_n(t)-\mathbf{E}\widehat{F}_n(t)\vert\right\}_{n\in \mathbb{N}}$, where $\widehat{F}_n(t)$ is the empirical distribution. Also, we find a similar sufficient condition for almost surely convergence.
\end{abstract}
\section{Introduction}

Empirical process is a fundamental topic in probability theory. Application of empirical process theory arises in many related fields, such as non-parametric statistics and statistical learning theory \cite{Mojirsheibani2007Nonparametric,Vaart2013Weak,Barrio2007Lectures,Dudley1967The,wainwright2019high}. While vigorous development of empirical process based on independent and identically distributed (i.i.d) random variables has been achieved by a large number of previous work \cite{Vaart2013Weak,Barrio2007Lectures}, few theoretical result has been provided when the independent condition is relaxed. Several work \cite{dedecker2007empirical,wu2004empirical,Dehling2002Empirical} studied the property of empirical process under weak dependence.\\

Different from i.i.d case, general dependence structure can be very complicated. Therefore, several studies \cite{Taqqu1977Law,arcones1994limit,delattre2016empirical,Dehling1989The,csorgo1996empirical,buchsteiner2018function} turned to some specific but common joint probability distribution structures, for example, multivariate Gaussian. In most of these studies, Hermite polynomials were adopted to deal with Gaussian random variables. We denote by $\phi(x)$ the density function of a standard Gaussian variable and by $\mu$ the standard Gaussian measure, then the Hermite polynomials $H_k(x)$ can be defined as
\begin{equation*}
    \phi^{(k)}(x)=(-1)^kH_k(x)\phi(x).
\end{equation*}
It is known that the normalized Hermite polynomials $\displaystyle\left\{\frac{H_k(x)}{\sqrt{k!}},k\geq 0\right\}$ form a Hilbert basis of the space $L^2(\sR, \mu)$, which is the space of square-integrable functions with respect to Gaussian measure. We let $\displaystyle h_k(x)=\frac{H_k(x)}{\sqrt{k!}}$. What's more, Hermite polynomials have another good property when it comes to bivariate Gaussian distribution. We denote by $(U,V)$ a centered bivariate Gaussian vector which obeys
\begin{equation*}
    (U,V)\sim \cN\left(0,\left(\begin{matrix}
    1 & \sigma\\
    \sigma & 1
    \end{matrix}\right)\right),
\end{equation*}
then we have \cite{Taqqu1977Law,delattre2016empirical}
\begin{equation}\label{cov}
    \mathbf{E}h_k(U)h_{k'}(V)=\sigma^k \delta_{k,k'},
\end{equation}
where $\delta$ is the Kronecker delta. This property offers us an opportunity to accurately interpret the dependence of multivariate Gaussian.\\

\newpage
In this paper, we adopt the chaining method \cite{Vaart2013Weak,Barrio2007Lectures,Dudley1967The,wainwright2019high} to build finite sample bounds for the empirical process of multivariate Gaussian. Since the index set of empirical distribution $\mathcal{C}=\{(-\infty,t]:t\in \sR\}$ is parameterized by a one-dimensional parameter $t$, the chaining method with $L^2$ norm is sufficient to yield a meaningful bound. Compared to metric space equipped with sub-Gaussian norm $\Vert \cdot\Vert_{\psi_2}$, metric space with $L^2$ norm $\Vert \cdot\Vert_2$ has more delicate algebra structure. We decompose the empirical process into the Hilbert basis $\{h_k(x),k\geq 0\}$, then the whole chaining method applied to the empirical process can be viewed as the chaining method applied to each subspace, which are orthogonal to each other. For the next step, the metric sum in the chaining method can be bounded by the quadratic variation of the projection on each subspace in some sense. Finally, we utilize the isometrically isomorph property of the Hilbert space $L^2(\sR,\mu)$ to calculate the aggregation of the quadratic variation on each subspace.\\

This paper is organized as follows. In Section \ref{fsb}, we present the meta  result Lemma \ref{lemma1}, which details the technique described above. For some technical reason, Lemma \ref{lemma1} deals with an empirical process $\widehat{Q}_n(t)$, which can be viewed as a smooth modification of the empirical distribution $\widehat{F}_n(t)$. Next in Theorem \ref{thm2}, we introduce a novel technique to build finite sample bound for the empirical distribution $\widehat{F}_n(t)$ by the result of $\widehat{Q}_n(t)$. Based on the results in Section \ref{fsb}, we present the main theorems in Section \ref{main} in advance. Theorem \ref{thm1} and Corollary \ref{cor1} about convergence in probability are direct corollaries of Theorem \ref{thm2}. Combine Theorem \ref{thm1} with \cite[Theorem 1]{Azriel2015The}, we have Corollary \ref{corequ} which states the condition given by Theorem \ref{thm1} is necessary and sufficient. After a more detailed discussion of the empirical distribution function, we also build Lemma \ref{lemas} and Theorem \ref{thmas} about almost surely convergence.\\

\noindent\textbf{Notation:} We let $\phi(x), \Phi(x),\mu$ be the density function of standard Gaussian, the cumulative function of standard Gaussian, and the standard Gaussian measure respectively. Given a measure $\nu$ on space $\mathbb{X}$, we denote $L^2(\mathbb{X},\nu)$ the space of square-integrable functions with respect to measure $\nu$. We let $\overset{L^2}{=}$ be the equality holds in the sense of certain $L^2$ space, and the specific $L^2$ space is clear in the context. $\lfloor \cdot \rfloor $ denotes the floor function. $\textbf{1}(A)$ denotes the indicator function with respect to event $A$.

\section{Main Results}\label{main}
In some realistic settings, we would like to ask how the elements in a stochastic process rather than an i.i.d sequence distribute in the long run. Some work has developed theories with the help of properties of certain dependence structures, including Markov property \cite{Kontorovich2012Uniform} and conditions regarding martingale difference \cite{dedecker2007empirical,wu2004empirical}.\\

In this section, we present our main results which show that jointly Gaussian is also a fundamental property. If the stochastic process is a Gaussian process, even under very general dependence structures, the empirical distribution regarding the elements of the process will converge. To define the empirical distribution concisely, we introduce the notion of standardized Gaussian process first.

\begin{defi}
A stochastic process $\{ X_k,k\in \mathbb{N}\}$ is called \textbf{Gaussian process} if and only if for every finite set of indices $\{k_1,k_2,\dots,k_t\}, t\geq 1$, the joint distribution of $(X_{k_1},X_{k_1},\dots,X_{k_t})$ is multivariate Gaussian. Furthermore, if $X_k\sim\cN(0,1)$ holds for every $k\in \mathbb{N}$, the process is called \textbf{standardized Gaussian process}.
\end{defi}

This paper focuses on the following empirical process which is defined by the components of a standardized Gaussian process
\begin{equation*}
    \widehat{F}_n(t)\overset{\triangle}{=}\frac{1}{n}\sum_{i=1}^n \textbf{1}(\Phi(X_i)\leq t),
\end{equation*}
where $\Phi(\cdot)$ is the cumulative distribution function of standard Gaussian. And, we define $\Delta(n)=2\sum\limits_{i< j\leq n}\vert Cov(X_i,X_j)\vert$ a dependence measure for the Gaussian process. Now we are ready to present the main theorem.

\begin{thm}\label{thm1}
Consider a standardized Gaussian process $\{ X_k,k\in \mathbb{N}\}$. Assume that the dependence measure satisfies

\begin{equation*}
    \lim_{n\rightarrow \infty} \frac{\Delta(n)}{n^2}=0.
\end{equation*}
Then we have 

\begin{equation*}
    \sup_t\vert\widehat{F}_n(t)-\mathbf{E}\widehat{F}_n(t) \vert\xrightarrow{P} 0.
\end{equation*}
\end{thm}

Standard Gaussian random variables (assumed joint normality) which satisfy $\Delta(n)=o(n^2)$ are also called weakly dependent normal variables \cite{Fan2010EstimatingFD}. Combine Theorem \ref{thm1} with \cite[Theorem 1]{Azriel2015The}, we have the following corollary which implies that \textbf{uniformly pointwise} convergence in probability is equivalent to \textbf{uniformly} convergence in probability for the empirical distribution of standardized Gaussian process.
\begin{cor}\label{corequ}
The following statements are equivalent.
\begin{itemize}
    \item $\{ X_k,k\in \mathbb{N}\}$ are weakly dependent normal variables;
    \item $\lim\limits_{n\rightarrow\infty}\sup\limits_tP(\vert\widehat{F}_n(t)-\mathbf{E}\widehat{F}_n(t) \vert>\epsilon)=0,\quad \forall \epsilon>0$;
    \item $\lim\limits_{n\rightarrow\infty}P(\sup\limits_t\vert\widehat{F}_n(t)-\mathbf{E}\widehat{F}_n(t) \vert>\epsilon)=0,\quad \forall \epsilon>0$.
\end{itemize}
\end{cor}

Note that for a sequence of completely identical standard Gaussian random variables, $\Delta(n)=O(n^2)$. Theorem \ref{thm1} implies that as long as the strong correlation condition is slightly relaxed, say, $\Delta(n)=o(n^2)$, the empirical distribution of the Gaussian random variable sequence will converge. This condition is met for a wide range of Gaussian process in realistic settings. For example, the Gaussian process whose covariance vanishes with the time shift. We summarize this result below.

\begin{cor}\label{cor1}
Consider a standardized Gaussian process $\{ X_k,k\in \mathbb{N}\}$. Suppose the covariance vanish with the time shift, that is to say, there exists a vanishing function $r(\cdot)$ with $\displaystyle\lim_{x\rightarrow \infty}r(x)=0$ and satisfying

\begin{equation*}
    \vert\mathbf{E}X_iX_j\vert \leq r(\vert i-j\vert).
\end{equation*}
Then we have 

\begin{equation*}
    \sup_t\vert\widehat{F}_n(t)-\mathbf{E}\widehat{F}_n(t) \vert\xrightarrow{P} 0.
\end{equation*}
\end{cor}

In addition to the results of convergence in probability, we are also able to consider almost surely convergence. We have the lemma below.

\begin{lem}\label{lemas}
Consider a standardized Gaussian process $\{ X_k,k\in \mathbb{N}\}$. Assume that the dependence measure satisfies

\begin{equation*}
    \sum_{i=1}^\infty \sqrt[3]{\frac{\Delta(\lfloor \gamma^i \rfloor)}{\lfloor \gamma^i \rfloor^2}}<+\infty,\quad\forall \gamma>1.
\end{equation*}
Then we have 

\begin{equation*}
    \sup_t\vert\widehat{F}_n(t)-\mathbf{E}\widehat{F}_n(t) \vert\xrightarrow{a.s.} 0.
\end{equation*}

\end{lem}

Note that Corollary \ref{corequ} actually implies the condition for convergence in probability in Theorem \ref{thm1} is necessary and sufficient, Lemma \ref{lemas} only offers a sufficient condition for almost surely convergence. However, Lemma \ref{lemas} gives the following theorem which states $\Delta(n)=O\left(n^2(\ln n)^{-3-\delta}\right)$, which only has a small gap with $\Delta(n)=o(n^2)$ in Theorem \ref{thm1}, is sufficient to ensure almost surely convergence.

\begin{thm}\label{thmas}
Consider a standardized Gaussian process $\{ X_k,k\in \mathbb{N}\}$. Assume that the dependence measure satisfies

\begin{equation*}
    \Delta(n)=O\left(n^2(\ln n)^{-3-\delta}\right)
\end{equation*}
for some $\delta>0$. Then we have 

\begin{equation*}
    \sup_t\vert\widehat{F}_n(t)-\mathbf{E}\widehat{F}_n(t) \vert\xrightarrow{a.s.} 0.
\end{equation*}

\end{thm}

Here we provide some examples where Theorem \ref{thm1}, Corollary \ref{cor1} and Theorem \ref{thmas} can be applied.\\

\noindent\textbf{Long-range dependence process} \; A standardized Gaussian process $\{ X_k,k\in \mathbb{N}\}$ is called long-range dependence process \cite{Taqqu1977Law} if 
\begin{equation*}
    \mathbf{E} X_iX_j=r(\vert i-j\vert),
\end{equation*}
where $r(0)=1, r(k)=k^{-D}L(k)$, $0<D<1$ and $L(\cdot)$ slowly varying at infinity. Then we have
\begin{equation*}
    \Delta(n)=o(n^{2-\frac{D}{2}}).
\end{equation*}

\noindent\textbf{Ornstein-Uhlenbeck process} \; We consider an Ornstein-Uhlenbeck process \cite{Durell1942The} defined by the following stochastic differential equation

\begin{equation*}
    dY_t=-\alpha Y_t dt+\sqrt{2\alpha}dW_t,\; Y_0\sim \mathcal{N}(0,1),
\end{equation*}
where $\alpha>0$ is a parameter and $W_t$ denotes the Wiener process. If we take $X_k=Y_k$ for $k\in \mathbb{N}$, then $\{ X_k,k\in \mathbb{N}\}$ is a standardized Gaussian process. The covariance function of $\{ X_k,k\in \mathbb{N}\}$ can be written as
\begin{equation*}
    \mathbf{E} X_iX_j=e^{-\alpha \vert i-j\vert}.
\end{equation*}
In this case, we have

\begin{equation*}
    \Delta(n)=O(n).
\end{equation*}

\section{Finite Sample Bounds for Multivariate Gaussian}\label{fsb}

In this section, we explain and illustrate in details the technique described in the introduction section. In Lemma \ref{lemma1}, we turn to consider an empirical process $\widehat{Q}_n(t)$, which can be viewed as a smooth version of empirical distribution $\widehat{F}_n(t)$. The smooth modification works in two aspects. First and foremost, the smoothness can ensure the sum of quadratic variation in different subspaces of the Hilbert basis to be finite. Secondly, continuity of the path saves us unnecessary trouble to consider limitation. 

\begin{lem}\label{lemma1}
Consider $X=(X_1,X_2,\dots,X_n)\sim \mathcal{N}(0,C)$ is a multivariate Gaussian random vector with covariance matrix $C$. Every element of $X$ has unit variance, that is to say, $C_{ii}=1$. Suppose $\ell$ is a continuously differentiable function with first order derivative supported on $[-\frac{1}{2},\frac{1}{2}]$. We define the following empirical process
\begin{equation*}
    \widehat{Q}_n(t)=\frac{1}{n}\sum_{i=1}^n \ell(t-\Phi(X_i)),
\end{equation*}
where $\Phi(\cdot)$ is the cumulative distribution function of standard Gaussian. Then we have

\begin{equation*}
    \mathbf{E}\sup_t\vert\widehat{Q}_n(t)-\mathbf{E}\widehat{Q}_n(t) \vert\leq (\sqrt{6}+\sqrt{3})D(\ell)\sqrt{\frac{n+\Delta}{n^2}},
\end{equation*}
where $\Delta=\sum\limits_{i\neq j}\vert C_{ij}\vert$ and 
\begin{equation*}
    D(\ell)=\sqrt{\int_{-1}^2 \left[\int_0^1 (\ell'(t-y))^2 dy-\left(\int_0^1 \ell'(t-y)dy\right)^2\right] dt}
\end{equation*}
is a functional which only depends on $\ell$.
\end{lem}

\begin{proof}
We write the expansion of $\ell(t-\Phi(x))$ in $L^2(\sR,\mu)$ as

\begin{equation}\label{expansion}
    \ell(t-\Phi(x))=\sum\limits_{k\geq 0} c_k(t)h_k(x),
\end{equation}
where
\begin{equation}
    c_k(t)=\int_\sR \ell(t-\Phi(x))h_k(x)d\mu.\label{ct}
\end{equation}
We denote a mean zero process
\begin{equation*}
    G(t)\overset{\triangle}{=}\widehat{Q}_n(t)-\mathbf{E}\widehat{Q}_n(t)=\frac{1}{n}\sum_{i=1}^n \ell(t-\Phi(X_i))-\mathbf{E}\left(\frac{1}{n}\sum_{i=1}^n \ell(t-\Phi(X_i))\right).
\end{equation*}
We denote by $\mu_C$ the Gaussian measure in $\sR^n$ defined by $\mathcal{N}(0,C)$. Then we can write the expansion of $G(t)$ in $L^2(\sR^n,\mu_C)$ by

\begin{equation}\label{Gt}
G(t)\overset{L^2}{=}\frac{1}{n}\sum_{i=1}^n\sum_{k\geq 0}c_k(t)h_k(X_i)-\mathbf{E}\frac{1}{n}\sum_{i=1}^n\sum_{k\geq 0}c_k(t)h_k(X_i),\quad \forall t.
\end{equation}
By considering $k'=0$ and $k> 0$ in Eq.~\eqref{cov}, we get $\mathbf{E}h_k(X_i)=0$ for $k>0$. What's more, by the definition of $h_0(x)$, we have $h_0(X_i)=\mathbf{E}h_0(X_i)=1$. Plug these results into Eq.~\eqref{Gt}, we get

\begin{equation}\label{Gtfinal}
G(t)\overset{L^2}{=}\frac{1}{n}\sum_{i=1}^n\sum_{k\geq 1}c_k(t)h_k(X_i),\quad \forall t.
\end{equation}
For simplicity, we let $c_k(s,t)=c_k(t)-c_k(s)$. Then by the expansion Eq.~\eqref{Gtfinal}, the second order increments of $G(t)$ can be bounded as

\begin{align}\nonumber
    \Vert G(t)-G(s) \Vert_2^2=& \mathbf{E}(G(t)-G(s))^2\\ \nonumber
    =& \mathbf{E}\left(\frac{1}{n}\sum_{i=1}^n\sum_{k\geq 1}c_k(t)h_k(X_i)-\frac{1}{n}\sum_{i=1}^n\sum_{k\geq 1}c_k(s)h_k(X_i)\right)^2\\ \nonumber
    =&\frac{1}{n^2}\sum_{i,j=1}^n \mathbf{E}\left(\sum_{k\geq 1}c_k(s,t)h_k(X_i)\right)\left(\sum_{k\geq 1}c_k(s,t)h_k(X_j)\right)\\ \nonumber
    =&\frac{1}{n^2}\sum_{i,j=1}^n \sum_{k\geq 1}c_k(s,t)^2(Cov(X_i,X_j))^k\\ 
    \leq &\frac{1}{n^2}\sum_{i,j=1}^n \sum_{k\geq 1}c_k(s,t)^2\vert Cov(X_i,X_j)\vert=\frac{1}{n^2}(n+\Delta)\sum_{k\geq 1}c_k(s,t)^2.\label{Gtincrement}
\end{align}
We consider equidistant $\frac{3}{2^m}$-nets $T_m$ of $[-1,2]$ for $m\in \mathbb{N}$. Then $\forall t\in[-1,2]$, there exists a sequence of points $\pi_m(t)\in T_m$ satisfying 
\begin{align*}
    &\pi_1(t)=-1,\\
    &\vert \pi_m(t)-\pi_{m+1}(t)\vert=\frac{3}{2^m},\\
    &\lim_{m\rightarrow{\infty}}\pi_m(t)=t.
\end{align*}
Since the path of $G(t)$ is continuous, we have

\begin{equation*}
    G(t)-G(-1)=\sum_{m\geq 1} G(\pi_{m+1}(t))-G(\pi_m(t)).
\end{equation*}
Keep in mind that $G(-1)=0$, then we have

\begin{align}\nonumber
    \mathbf{E}\sup_t\vert G(t)\vert&=\mathbf{E}\sup_t\vert\sum_{m\geq 1} G(\pi_{m+1}(t))-G(\pi_m(t))\vert\\ \nonumber
    &\leq \mathbf{E}\sup_t\sum_{m\geq 1}\vert G(\pi_{m+1}(t))-G(\pi_m(t))\vert\\ \nonumber
    &\leq \mathbf{E}\sum_{m\geq 1}\sup_t\vert G(\pi_{m+1}(t))-G(\pi_m(t))\vert\\
    &=\sum_{m\geq 1}\mathbf{E}\sup_t\vert G(\pi_{m+1}(t))-G(\pi_m(t))\vert.\label{chaining1}
\end{align}
Since $\pi_m(t)\in T_m$, $\pi_{m+1}(t)\in T_{m+1}$ and $\vert \pi_m(t)-\pi_{m+1}(t)\vert=\frac{3}{2^m}$, the expectation of supremum can be bounded as
\begin{align*}
    (\mathbf{E}\sup_t\vert G(\pi_{m+1}(t))-G(\pi_m(t))\vert)^2\leq&\mathbf{E}(\sup_t\vert G(\pi_{m+1}(t))-G(\pi_m(t))\vert)^2\\
    \leq & \mathbf{E}\sum_{a=1}^{2^m}  (G(\frac{3a}{2^m}-1)-G(\frac{3(a-1)}{2^m}-1))^2     \\
    \leq & \sum_{a=1}^{2^m}\frac{1}{n^2} (n+\Delta) \sum_{k\geq 1} c_k(\frac{3(a-1)}{2^m}-1,\frac{3a}{2^m}-1)^2
\end{align*}
On the other hand, by Cauchy-Schwarz inequality, we have

\begin{equation*}
    c_k(\frac{3(a-1)}{2^m}-1,\frac{3a}{2^m}-1)^2=\left(\int_{\frac{3(a-1)}{2^m}-1}^{\frac{3a}{2^m}-1} c_k'(t)dt\right)^2\leq \left(\int_{\frac{3(a-1)-1}{2^m}}^{\frac{3a}{2^m}-1} 1dt\right) \left(\int_{\frac{3(a-1)}{2^m}-1}^{\frac{3a}{2^m}-1} (c_k'(t))^2dt\right).
\end{equation*}
Thus,
\begin{align}\nonumber
    (\mathbf{E}\sup_t\vert G(\pi_{m+1}(t))-G(\pi_m(t))\vert)^2\leq& \sum_{a=1}^{2^m}\frac{1}{n^2} (n+\Delta) \sum_{k\geq 1} c_k(\frac{3(a-1)}{2^m}-1,\frac{3a}{2^m}-1)^2\\ \nonumber
    = & \frac{n+\Delta}{n^2}\sum_{k\geq 1} \sum_{a=1}^{2^m} c_k(\frac{3(a-1)}{2^m}-1,\frac{3a}{2^m}-1)^2\\ \nonumber
    \leq & \frac{n+\Delta}{n^2}\sum_{k\geq 1} \sum_{a=1}^{2^m}\frac{3}{2^m} \left(\int_{\frac{3(a-1)}{2^m}-1}^{\frac{3a}{2^m}-1} (c_k'(t))^2dt\right)\\
    =& \frac{3}{2^m}\frac{n+\Delta}{n^2}\sum_{k\geq 1}\int_{-1}^2(c_k'(t))^2dt.\label{chaining2}
\end{align}
Combine Eq.~\eqref{chaining1} and Eq.~\eqref{chaining2}, we have
\begin{align}\nonumber
    \mathbf{E}\sup_t\vert G(t)\vert \leq &\sum_{m\geq 1}\sqrt{\frac{3}{2^m}\frac{n+\Delta}{n^2}\sum_{k\geq 1}\int_{-1}^2(c_k'(t))^2dt}\\ 
    =& (\sqrt{6}+\sqrt{3})\sqrt{\frac{n+\Delta}{n^2}}\sqrt{\sum_{k\geq 1}\int_{-1}^2(c_k'(t))^2dt}.\label{Gtbound1}
\end{align}
Finally, let's take a close look at $\sum\limits_{k\geq 1}\int_{-1}^2(c_k'(t))^2dt$. Since $\ell'$ is continuous and supported on a compact set, we have 
\begin{equation*}
    \ell'(t-\Phi(x))\overset{L^2}{=}\sum_{k\geq 0}c_k'(t)h_k(x).
\end{equation*}
Therefore, 
\begin{align*}
    \sum_{k\geq 0}(c_k'(t))^2&=\int_\sR \left(\sum_{k\geq 0}c_k'(t)h_k(x)\right)\left(\sum_{k\geq 0}c_k'(t)h_k(x)\right)d\mu \\
    &= \int_\sR (\ell'(t-\Phi(x)))^2d\mu \\
    &=\int_0^1 (\ell'(t-y))^2 dy.
\end{align*}
On the other hand, by Eq.~\eqref{ct} we have

\begin{equation*}
    c_0'(t)=\left(\int_\sR \ell(t-\Phi(x))d\mu\right)'=\int_\sR \ell'(t-\Phi(x))d\mu=\int_0^1 \ell'(t-y)dy.
\end{equation*}
Thus,

\begin{equation}
    \sum_{k\geq 1}(c_k'(t))^2=\int_0^1 (\ell'(t-y))^2 dy-\left(\int_0^1 \ell'(t-y)dy\right)^2. \label{aggregation}
\end{equation}
Combine Eq.~\eqref{Gtbound1} and Eq.~\eqref{aggregation}, we have 

\begin{equation*}
    \mathbf{E}\sup_t\vert G(t)\vert \leq (\sqrt{6}+\sqrt{3})D(\ell)\sqrt{\frac{n+\Delta}{n^2}},
\end{equation*}
where $D(\ell)$ is a functional of $\ell(\cdot)$:

\begin{equation*}
    D(\ell)^2=\int_{-1}^2 \left[\int_0^1 (\ell'(t-y))^2 dy-\left(\int_0^1 \ell'(t-y)dy\right)^2\right] dt.
\end{equation*}
\end{proof}

\begin{rem}
We would like to point out that the rate regarding $\Delta$ in Lemma \ref{lemma1} is optimal. For $\Delta=\Omega(n)$, we choose the covariance matrix $C$ by

\begin{equation*}
    C_{ij}=\left\{
\begin{array}{rcl}
1     &      & i,j\leq \lfloor \frac{1+\sqrt{1+4\Delta}}{2} \rfloor \\
1     &      & i=j \\
\xi   &      & otherwise,
\end{array} \right.
\end{equation*}
where $\xi$ ensures the equality $\Delta=\sum\limits_{i\neq j}\vert C_{ij}\vert$ holds. That is to say, there are $\lfloor \frac{1+\sqrt{1+4\Delta}}{2} \rfloor$ elements of the multivariate Gaussian $X\sim \mathcal{N}(0,C)$ take the same value. Thus, with high probability we have

\begin{equation*}
    \sup_t\vert\widehat{Q}_n(t)-\mathbf{E}\widehat{Q}_n(t) \vert=\Omega\left(\frac{1}{n}\sqrt{\Delta}\right).
\end{equation*}
\end{rem}

Now we are going to deduce the finite sample bound for empirical distribution $\widehat{F}_n(t)$. Instead of bounding the difference $\displaystyle \sup_t\vert \widehat{Q}_n(t)-\widehat{F}_n(t)\vert$ directly, we translate the process $\widehat{Q}_n(t)$ on the index set and squeeze $\widehat{F}_n(t)$ by $\widehat{Q}_n(t-\epsilon)$ and $\widehat{Q}_n(t+\epsilon)$.

\begin{thm}\label{thm2}
Consider $X=(X_1,X_2,\dots,X_n)\sim \mathcal{N}(0,C)$ is a multivariate Gaussian random vector with covariance matrix $C$. Every element of $X$ has unit variance, that is to say, $C_{ii}=1$. 
We consider the empirical distribution
\begin{equation*}
    \widehat{F}_n(t)=\frac{1}{n}\sum_{i=1}^n \textbf{1}(\Phi(X_i)\leq t),
\end{equation*}
where $\Phi(\cdot)$ is the cumulative distribution function of standard Gaussian. Then we have

\begin{equation*}
    \mathbf{E}\sup_t\vert\widehat{F}_n(t)-\mathbf{E}\widehat{F}_n(t) \vert\leq 16\sqrt[3]{\frac{n+\Delta}{n^2}},
\end{equation*}
where $\Delta=\sum\limits_{i\neq j}\vert C_{ij}\vert$.
\end{thm}

\begin{proof}
We choose a $\ell$ in Lemma \ref{lemma1} in the following way:
\begin{align*}
    &\ell'(x)=\frac{1}{\epsilon^2}(\epsilon-\vert x\vert)^+,\\
    &\ell(-\frac{1}{2})=0.
\end{align*}
Where $\epsilon$ is a parameter less than $\frac{1}{2}$. It is easy to calculate that
\begin{equation}
    D(\ell)\leq \sqrt{3\frac{2\epsilon^3}{3\epsilon^4}}=\sqrt{\frac{2}{\epsilon}}.\label{dl}
\end{equation}
And, since $\vert\ell(\cdot)-\textbf{1}(\cdot\geq 0)\vert$ is only supported on $[-\epsilon,\epsilon]$ and is bounded by $\frac{1}{2}$, we have
\begin{align}
    \vert \mathbf{E} \widehat{Q}_n(t)-\mathbf{E}\widehat{F}_n(t)\vert&=\vert \int_0^1 \ell(t-y)dy-\textbf{1}(t-y\geq 0)dy\vert \leq \frac{1}{2}2\epsilon=\epsilon. \label{expdifference}
\end{align}
Thus by Lemma \ref{lemma1}, Eq.~\eqref{dl}, and Eq.~\eqref{expdifference} we have

\begin{equation}
    \mathbf{E}\sup_t\vert \widehat{Q}_n(t)-\mathbf{E}\widehat{F}_n(t)\vert \leq \mathbf{E}\sup_t\vert \widehat{Q}_n(t)-\mathbf{E}\widehat{Q}_n(t)\vert +\sup_t \vert \mathbf{E} \widehat{Q}_n(t)-\mathbf{E}\widehat{F}_n(t)\vert\leq 6\sqrt{\frac{n+\Delta}{\epsilon n^2}}+\epsilon.\label{ori}
\end{equation}
One simple variation of Eq.~\eqref{ori} is 

\begin{align}\label{+}
    \mathbf{E}\sup_t( \widehat{Q}_n(t)-\mathbf{E}\widehat{F}_n(t))^+ &\leq 6\sqrt{\frac{n+\Delta}{\epsilon n^2}}+\epsilon,\\ 
    \mathbf{E}\sup_t( \widehat{Q}_n(t)-\mathbf{E}\widehat{F}_n(t))^- &\leq 6\sqrt{\frac{n+\Delta}{\epsilon n^2}}+\epsilon.\label{-}
\end{align}
By the definition of $\ell$, we have

\begin{equation*}
    \ell(x-\epsilon)\leq \textbf{1}(x\geq 0) \leq \ell(x+\epsilon).
\end{equation*}
Thus,

\begin{equation*}
    \widehat{Q}_n(t-\epsilon)\leq \widehat{F}_n(t)\leq \widehat{Q}_n(t+\epsilon).
\end{equation*}
By plugging $t+\epsilon$ and $t-\epsilon$ into Eq.~\eqref{+} and Eq.~\eqref{-} respectively, we have

\begin{align*}
    \mathbf{E}\sup_t( \widehat{F}_n(t)-\mathbf{E}\widehat{F}_n(t+\epsilon))^+\leq \mathbf{E}\sup_t(\widehat{Q}_n(t+\epsilon)-\mathbf{E}\widehat{F}_n(t+\epsilon))^+ &\leq 6\sqrt{\frac{n+\Delta}{\epsilon n^2}}+\epsilon,\\ 
    \mathbf{E}\sup_t( \widehat{F}_n(t)-\mathbf{E}\widehat{F}_n(t-\epsilon))^-\leq \mathbf{E}\sup_t(\widehat{Q}_n(t-\epsilon)-\mathbf{E}\widehat{F}_n(t-\epsilon))^- &\leq 6\sqrt{\frac{n+\Delta}{\epsilon n^2}}+\epsilon.
\end{align*}
What's more, by the definition of $\widehat{F}_n(t)$ we have

\begin{align*}
    \mathbf{E}\widehat{F}_n(t+\epsilon)&\leq \mathbf{E}\widehat{F}_n(t)+\epsilon,\\
    \mathbf{E}\widehat{F}_n(t-\epsilon)&\geq \mathbf{E}\widehat{F}_n(t)-\epsilon.
\end{align*}
As a result, we have

\begin{align*}
    \mathbf{E}\sup_t( \widehat{F}_n(t)-\mathbf{E}\widehat{F}_n(t))^+&\leq \mathbf{E}\sup_t( \widehat{F}_n(t)-\mathbf{E}\widehat{F}_n(t+\epsilon)+\epsilon)^+\\
    & \leq \mathbf{E}\sup_t( \widehat{F}_n(t)-\mathbf{E}\widehat{F}_n(t+\epsilon))^++\epsilon\leq 6\sqrt{\frac{n+\Delta}{\epsilon n^2}}+2\epsilon,\\
    \mathbf{E}\sup_t( \widehat{F}_n(t)-\mathbf{E}\widehat{F}_n(t))^-&\leq \mathbf{E}\sup_t( \widehat{F}_n(t)-\mathbf{E}\widehat{F}_n(t-\epsilon)-\epsilon)^-\\
    & \leq \mathbf{E}\sup_t( \widehat{F}_n(t)-\mathbf{E}\widehat{F}_n(t+\epsilon))^-+\epsilon\leq 6\sqrt{\frac{n+\Delta}{\epsilon n^2}}+2\epsilon.
\end{align*}
Combine the two inequalities above, we get

\begin{align}\nonumber
    \mathbf{E}\sup_t\vert \widehat{F}_n(t)-\mathbf{E}\widehat{F}_n(t)\vert&\leq\mathbf{E}(\sup_t( \widehat{F}_n(t)-\mathbf{E}\widehat{F}_n(t))^++\sup_t( \widehat{F}_n(t)-\mathbf{E}\widehat{F}_n(t))^-)\\ \nonumber
    &=\mathbf{E}\sup_t( \widehat{F}_n(t)-\mathbf{E}\widehat{F}_n(t))^++\mathbf{E}\sup_t( \widehat{F}_n(t)-\mathbf{E}\widehat{F}_n(t))^-\\
    &\leq  12\sqrt{\frac{n+\Delta}{\epsilon n^2}}+4\epsilon.\label{fb}
\end{align}
When $\frac{n+\Delta}{n^2}\leq \frac{1}{18}$, by taking $\epsilon=\sqrt[3]{\frac{9(n+\Delta)}{4n^2}}$ in Eq.~\eqref{fb}, we get
\begin{equation*}
    \mathbf{E}\sup_t\vert \widehat{F}_n(t)-\mathbf{E}\widehat{F}_n(t)\vert\leq  16\sqrt[3]{\frac{n+\Delta}{n^2}}.
\end{equation*}
Otherwise, if $\frac{n+\Delta}{n^2}> \frac{1}{18}$, we simply have
\begin{equation*}
    \mathbf{E}\sup_t\vert \widehat{F}_n(t)-\mathbf{E}\widehat{F}_n(t)\vert\leq 1\leq 16\sqrt[3]{\frac{n+\Delta}{n^2}}.
\end{equation*}
To sum up, we finally have
\begin{equation*}
    \mathbf{E}\sup_t\vert \widehat{F}_n(t)-\mathbf{E}\widehat{F}_n(t)\vert\leq  16\sqrt[3]{\frac{n+\Delta}{n^2}}.
\end{equation*}
\end{proof}

\section{Proof of Main Results}
\begin{rthm}
Consider a standardized Gaussian process $\{ X_k,k\in \mathbb{N}\}$. Assume that the dependence measure satisfies

\begin{equation*}
    \lim_{n\rightarrow \infty} \frac{\Delta(n)}{n^2}=0.
\end{equation*}
Then we have 

\begin{equation*}
    \sup_t\vert\widehat{F}_n(t)-\mathbf{E}\widehat{F}_n(t) \vert\xrightarrow{P} 0.
\end{equation*}
\end{rthm}
\begin{proof}
Since $\displaystyle\lim_{n\rightarrow \infty} \frac{\Delta(n)}{n^2}=0$, by Theorem \ref{thm2} we have 
\begin{equation*}
    \lim_{n\rightarrow\infty} \mathbf{E} \sup_t\vert\widehat{F}_n(t)-\mathbf{E}\widehat{F}_n(t)\vert=0.
\end{equation*}
Then for all $\epsilon>0$, by Markov inequality, we have

\begin{equation*}
     \lim_{n\rightarrow\infty} P(\sup_t\vert\widehat{F}_n(t)-\mathbf{E}\widehat{F}_n(t)\vert\geq \epsilon)\leq \lim_{n\rightarrow\infty}\frac{\mathbf{E} \sup_t\vert\widehat{F}_n(t)-\mathbf{E}\widehat{F}_n(t)\vert}{\epsilon}=0.
\end{equation*}
That is to say,
\begin{equation*}
    \sup_t\vert\widehat{F}_n(t)-\mathbf{E}\widehat{F}_n(t) \vert\xrightarrow{P} 0.
\end{equation*}

\end{proof}

\begin{rcor}
The following statements are equivalent.
\begin{itemize}
    \item $\{ X_k,k\in \mathbb{N}\}$ are weakly dependent normal variables;
    \item $\lim\limits_{n\rightarrow\infty}\sup\limits_tP(\vert\widehat{F}_n(t)-\mathbf{E}\widehat{F}_n(t) \vert>\epsilon)=0,\quad \forall \epsilon>0$;
    \item $\lim\limits_{n\rightarrow\infty}P(\sup\limits_t\vert\widehat{F}_n(t)-\mathbf{E}\widehat{F}_n(t) \vert>\epsilon)=0,\quad \forall \epsilon>0$.
\end{itemize}
\end{rcor}
\begin{proof}
\cite[Theorem 1]{Azriel2015The} states that
\begin{equation}\label{quotethm}
\lim_{n\rightarrow \infty}\frac{\Delta(n)}{n^2}=0\Longleftrightarrow \lim_{n\rightarrow \infty}\sup_t\mathbf{E}\vert\widehat{F}_n(t)-\mathbf{E}\widehat{F}_n(t) \vert^2=0.
\end{equation}
As the first step, we have 

\begin{align*}
\lim_{n\rightarrow\infty}\sup_t P(\vert\widehat{F}_n(t)-\mathbf{E}\widehat{F}_n(t) \vert>\epsilon)\leq \lim_{n\rightarrow\infty}\sup_t\frac{1}{\epsilon^2}\mathbf{E}\vert\widehat{F}_n(t)-\mathbf{E}\widehat{F}_n(t) \vert^2.
\end{align*}
Thus $\displaystyle \lim_{n\rightarrow \infty}\sup_t\mathbf{E}\vert\widehat{F}_n(t)-\mathbf{E}\widehat{F}_n(t) \vert^2=0\Longrightarrow \lim\limits_{n\rightarrow\infty}\sup\limits_tP(\vert\widehat{F}_n(t)-\mathbf{E}\widehat{F}_n(t) \vert>\epsilon)=0,\quad \forall \epsilon>0$. On the other hand, if we assume the later statement holds, since $\vert\widehat{F}_n(t)-\mathbf{E}\widehat{F}_n(t) \vert\leq 1$, then for any given $\epsilon$ we have

\begin{align*}
\lim_{n\rightarrow \infty}\sup_t\mathbf{E}\vert\widehat{F}_n(t)-\mathbf{E}\widehat{F}_n(t) \vert^2\leq \lim_{n\rightarrow \infty} \sup_{t}\left(\epsilon^2+P(\vert\widehat{F}_n(t)-\mathbf{E}\widehat{F}_n(t) \vert>\epsilon)\right)=\epsilon^2.
\end{align*}
By the arbitrariness of $\epsilon$, we have $\displaystyle\lim_{n\rightarrow \infty}\sup_t\mathbf{E}\vert\widehat{F}_n(t)-\mathbf{E}\widehat{F}_n(t) \vert^2=0$. That is to say, 

\begin{equation}\label{l2pequivalence}
\lim_{n\rightarrow \infty}\sup_t\mathbf{E}\vert\widehat{F}_n(t)-\mathbf{E}\widehat{F}_n(t) \vert^2=0\Longleftrightarrow \lim_{n\rightarrow\infty}\sup_tP(\vert\widehat{F}_n(t)-\mathbf{E}\widehat{F}_n(t) \vert>\epsilon)=0,\quad \forall \epsilon>0.
\end{equation}
Combine Eq.~\eqref{quotethm} and Eq.~\eqref{l2pequivalence}, we get 

\begin{equation}\label{12equivalence}
\lim_{n\rightarrow \infty}\frac{\Delta(n)}{n^2}=0\Longleftrightarrow \lim_{n\rightarrow\infty}\sup_tP(\vert\widehat{F}_n(t)-\mathbf{E}\widehat{F}_n(t) \vert>\epsilon)=0,\quad \forall \epsilon>0.
\end{equation}
Theorem \ref{thm1} states that

\begin{equation}\label{logic1}
\lim_{n\rightarrow \infty}\frac{\Delta(n)}{n^2}=0\Longrightarrow \lim_{n\rightarrow\infty}P(\sup_t\vert\widehat{F}_n(t)-\mathbf{E}\widehat{F}_n(t) \vert>\epsilon)=0,\quad \forall \epsilon>0.
\end{equation}
And it is obvious that

\begin{equation}\label{logic2}
\lim_{n\rightarrow\infty}\sup_tP(\vert\widehat{F}_n(t)-\mathbf{E}\widehat{F}_n(t) \vert>\epsilon)\leq \lim_{n\rightarrow\infty}P(\sup_t\vert\widehat{F}_n(t)-\mathbf{E}\widehat{F}_n(t) \vert>\epsilon),\quad \forall \epsilon>0.
\end{equation}
Combine Eq.~\eqref{12equivalence}, Eq.~\eqref{logic1} and Eq.~\eqref{logic2}, we can conclude that the statements in Corollary \ref{corequ} are equivalent.
\end{proof}

\begin{rcor}
Consider a standardized Gaussian process $\{ X_k,k\in \mathbb{N}\}$. Suppose the covariance vanish with the time shift, that is to say, there exists a vanishing function $r(\cdot)$ with $\displaystyle\lim_{x\rightarrow \infty}r(x)=0$ and satisfying

\begin{equation*}
    \vert\mathbf{E}X_iX_j\vert \leq r(\vert i-j\vert).
\end{equation*}
Then we have 

\begin{equation*}
    \sup_t\vert\widehat{F}_n(t)-\mathbf{E}\widehat{F}_n(t) \vert\xrightarrow{P} 0.
\end{equation*}
\end{rcor}

\begin{proof}
We consider the dependence measure $\Delta(n)$ in Theorem \ref{thm1}. By the definition of limitation, $\forall \delta>0$, there exists $N(\delta)\in \mathbb{N}$, s.t., for all $n\geq N(\delta)$, we have $r(n)\leq \frac{\delta}{2}$. Then for all $n> \left(1-\sqrt{1-\frac{\delta}{2}}\right)^{-1}N(\delta)$, we have

\begin{align*}
    \frac{\Delta(n)}{n^2}\leq \frac{\delta}{2}\frac{(n-N(\delta))(n-N(\delta)+1)}{n^2}+\frac{n^2-(n-N(\delta))^2}{n^2}\leq \delta.
\end{align*}
Again by the definition of limitation, we have $\lim_{n\rightarrow\infty}\frac{\Delta(n)}{n^2}=0$. By Theorem \ref{thm1}, we have
\begin{equation*}
    \sup_t\vert\widehat{F}_n(t)-\mathbf{E}\widehat{F}_n(t) \vert\xrightarrow{P} 0.
\end{equation*}
\end{proof}
In order to achieve almost surely convergence, it is equivalent to prove that

\begin{equation}\label{asconvergence}
\lim_{m\rightarrow \infty} P\left(\bigcup_{n=m}^\infty\{\sup_t\vert\widehat{F}_n(t)-\mathbf{E}\widehat{F}_n(t)\vert>\epsilon\}\right)=0, \quad \forall \epsilon>0. 
\end{equation}
Unfortunately, the finite sample bound in Theorem \ref{thm2} is not sufficient to derive Eq.~\eqref{asconvergence} directly, since the summation $\displaystyle \sum_{i=1}^\infty \sqrt[3]{\frac{n+\Delta(n)}{n^2}}$ always diverges. However, one may notice that the fluctuation in the sequence $\displaystyle\left\{\sup_t\vert\widehat{F}_n(t)-\mathbf{E}\widehat{F}_n(t)\vert\right\}_{n\in \mathbb{N}}$ is very small. Denote $D_n=\sup\limits_t\vert\widehat{F}_n(t)-\mathbf{E}\widehat{F}_n(t)\vert$, we have

\begin{align*}
D_n-D_{n+1}=&\sup_t\vert\widehat{F}_n(t)-\mathbf{E}\widehat{F}_n(t)\vert-\sup_t\vert\widehat{F}_{n+1}(t)-\mathbf{E}\widehat{F}_{n+1}(t)\vert\\
\leq & \vert\widehat{F}_n(t_n)-\mathbf{E}\widehat{F}_n(t_n)\vert-\vert\widehat{F}_{n+1}(t_n)-\mathbf{E}\widehat{F}_{n+1}(t_n)\vert\\
\leq & \vert\widehat{F}_n(t_n)-\widehat{F}_{n+1}(t_n)\vert\\
= & \vert \frac{1}{n(n+1)}\sum_{i=1}^n \textbf{1}(\Phi(X_i)\leq t_n)-\frac{1}{n+1}\textbf{1}(\Phi(X_{n+1})\leq t_n)\vert\\
\leq & \max\{\frac{1}{n(n+1)}\sum_{i=1}^n \textbf{1}(\Phi(X_i)\leq t_n), \frac{1}{n+1}\textbf{1}(\Phi(X_{n+1})\leq t_n)\}\leq \frac{1}{n+1},
\end{align*}
where $t_n=\mathop{\arg\max}\limits_{t}\vert\widehat{F}_n(t)-\mathbf{E}\widehat{F}_n(t)\vert$. In a similar way, we have 

\begin{align*}
D_{n+1}-D_n=&\sup_t\vert\widehat{F}_{n+1}(t)-\mathbf{E}\widehat{F}_{n+1}(t)\vert-\sup_t\vert\widehat{F}_n(t)-\mathbf{E}\widehat{F}_n(t)\vert\\
\leq & \vert\widehat{F}_{n+1}(t_{n+1})-\mathbf{E}\widehat{F}_{n+1}(t_{n+1})\vert-\vert\widehat{F}_n(t_{n+1})-\mathbf{E}\widehat{F}_n(t_{n+1})\vert\\
\leq & \vert\widehat{F}_{n+1}(t_{n+1})-\widehat{F}_{n}(t_{n+1})\vert\\
= & \vert \frac{1}{n+1}\textbf{1}(\Phi(X_{n+1})\leq t_{n+1})-\frac{1}{n(n+1)}\sum_{i=1}^n \textbf{1}(\Phi(X_i)\leq t_{n+1})\vert\\
\leq & \max\{\frac{1}{n+1}\textbf{1}(\Phi(X_{n+1})\leq t_{n+1}), \frac{1}{n(n+1)}\sum_{i=1}^n \textbf{1}(\Phi(X_i)\leq t_{n+1})\}\leq \frac{1}{n+1},
\end{align*}
where $t_{n+1}=\mathop{\arg\max}\limits_{t}\vert\widehat{F}_{n+1}(t)-\mathbf{E}\widehat{F}_{n+1}(t)\vert$. As a conclusion, we get 

\begin{equation*}
\vert D_n-D_{n+1}\vert \leq \frac{1}{n+1},
\end{equation*}
which implies

\begin{equation*}
\vert D_m-D_n\vert =O\left(\ln \frac{n}{m} \right).
\end{equation*}
In this way, the event in Eq.~\eqref{asconvergence} can be covered by a union of events with exponentially increasing indexes, a fact which is fundamental to the proof of almost surely convergence.

\begin{rlem}
Consider a standardized Gaussian process $\{ X_k,k\in \mathbb{N}\}$. Assume that the dependence measure satisfies

\begin{equation*}
    \sum_{i=1}^\infty \sqrt[3]{\frac{\Delta(\lfloor \gamma^i \rfloor)}{\lfloor \gamma^i \rfloor^2}}<+\infty,\quad\forall \gamma>1.
\end{equation*}
Then we have 

\begin{equation*}
    \sup_t\vert\widehat{F}_n(t)-\mathbf{E}\widehat{F}_n(t) \vert\xrightarrow{a.s.} 0.
\end{equation*}

\end{rlem}

\begin{proof}
We consider a sequence of indexes $j(i)=\lfloor \gamma^i \rfloor$ for a given $\gamma$. Then $\forall \; j(i)< k< j(i+1)$, we have

\begin{align*}
\vert D_k-D_{j(i)}\vert\leq \sum_{p=j(i)}^{k-1}\frac{1}{p+1}\leq \frac{\lfloor \gamma^{i+1} \rfloor-1-\lfloor \gamma^i \rfloor}{\lfloor \gamma^i \rfloor+1}<\frac{\gamma^{i+1}-\gamma^i}{\gamma^i}=\gamma-1.
\end{align*}
One should notice that $j(i)< k< j(i+1)$ implies $j(i+1)-j(i)\geq 2$ here. By taking $\gamma=1+\frac{\epsilon}{2}$, we have

\begin{equation*}
\bigcup_{n=m}^\infty \{D_n> \epsilon\}\subset\bigcup_{i=\widetilde{i}(m)}^\infty \{D_{j(i)}>\frac{\epsilon}{2}\},
\end{equation*}
where $\widetilde{i}(m)=\max\{i:j(i)\leq m \}$. Therefore, the probability in Eq.~\eqref{asconvergence} can be bounded as

\begin{equation*}
P\left(\bigcup_{n=m}^\infty \{D_n> \epsilon\}\right)\leq P\left(\bigcup_{i=\widetilde{i}(m)}^\infty \{D_{j(i)}>\frac{\epsilon}{2}\}\right)\leq \sum_{i=\widetilde{i}(m)}^\infty P\left(D_{j(i)}>\frac{\epsilon}{2}\right).
\end{equation*}
Since $\displaystyle \lim_{m\rightarrow \infty}\widetilde{i}(m)=\infty$, we only need $\displaystyle\sum_{i=1}^\infty P\left(D_{j(i)}>\frac{\epsilon}{2}\right) <+ \infty$. Thus we get a sufficient condition

\begin{equation}\label{condition1}
\sum_{i=1}^\infty P\left(D_{j(i)}>\frac{\epsilon}{2}\right) <+ \infty, \quad \forall \epsilon>0.
\end{equation}
Combine Eq.~\eqref{condition1} with Theorem \ref{fsb} and Markov inequality, we have 

\begin{equation*}
    \sum_{i=1}^\infty P\left(D_{j(i)}>\frac{\epsilon}{2}\right)\leq \frac{32}{\epsilon}\sum_{i=1}^\infty \sqrt[3]{\frac{\lfloor \gamma^i \rfloor+\Delta(\lfloor \gamma^i \rfloor)}{\lfloor \gamma^i \rfloor^2}}.
\end{equation*}
Notice that 

\begin{equation*}
\sum_{i=1}^\infty \sqrt[3]{\frac{\Delta(\lfloor \gamma^i \rfloor)}{\lfloor \gamma^i \rfloor^2}}<\sum_{i=1}^\infty \sqrt[3]{\frac{\lfloor \gamma^i \rfloor+\Delta(\lfloor \gamma^i \rfloor)}{\lfloor \gamma^i \rfloor^2}}<\sum_{i=1}^\infty \left(\sqrt[3]{\frac{1}{\lfloor \gamma^i \rfloor}}+\sqrt[3]{\frac{\Delta(\lfloor \gamma^i \rfloor)}{\lfloor \gamma^i \rfloor^2}}\right)\leq \frac{\sqrt[3]{2}}{\sqrt[3]{\gamma}-1} +\sum_{i=1}^\infty \sqrt[3]{\frac{\Delta(\lfloor \gamma^i \rfloor)}{\lfloor \gamma^i \rfloor^2}},
\end{equation*}
we have

\begin{equation*}
\sum_{i=1}^\infty \sqrt[3]{\frac{\Delta(\lfloor \gamma^i \rfloor)}{\lfloor \gamma^i \rfloor^2}}<+\infty\Longleftrightarrow\sum_{i=1}^\infty \sqrt[3]{\frac{\lfloor \gamma^i \rfloor+\Delta(\lfloor \gamma^i \rfloor)}{\lfloor \gamma^i \rfloor^2}}<+\infty, \quad \forall \gamma>1.
\end{equation*}
To sum up, we finally get the sufficient condition
\begin{equation*}
    \sum_{i=1}^\infty \sqrt[3]{\frac{\Delta(\lfloor \gamma^i \rfloor)}{\lfloor \gamma^i \rfloor^2}}<+\infty,\quad\forall \gamma>1.
\end{equation*}
\end{proof}

\begin{rthm}\label{thmas}
Consider a standardized Gaussian process $\{ X_k,k\in \mathbb{N}\}$. Assume that the dependence measure satisfies

\begin{equation*}
    \Delta(n)=O\left(n^2(\ln n)^{-3-\delta}\right)
\end{equation*}
for some $\delta>0$. Then we have 

\begin{equation*}
    \sup_t\vert\widehat{F}_n(t)-\mathbf{E}\widehat{F}_n(t) \vert\xrightarrow{a.s.} 0.
\end{equation*}

\end{rthm}
\begin{proof}
We only need to check that the condition in Lemma \ref{lemas} is satisfied. $\forall \gamma>1$, we have

\begin{align*}
\sum_{i=1}^\infty \sqrt[3]{\frac{\Delta(\lfloor \gamma^i \rfloor)}{\lfloor \gamma^i \rfloor^2}}&= O\left(\lfloor\log_{\gamma}2\rfloor+\sum_{i=\lfloor\log_{\gamma}2\rfloor+1}^\infty \sqrt[3]{\frac{\lfloor \gamma^i \rfloor^2\left(\ln \lfloor \gamma^i \rfloor\right)^{-3-\delta}}{\lfloor \gamma^i \rfloor^2}}\right)\\
&\leq O\left(\lfloor\log_{\gamma}2\rfloor+\sum_{i=1}^\infty\frac{1}{(i\ln \gamma)^{1+\frac{\delta}{3}}}\right)=O\left(\lfloor\log_{\gamma}2\rfloor+\frac{1}{(\ln\gamma)^{1+\frac{\delta}{3}}}\right)< +\infty.
\end{align*}

\end{proof}

\section{Discussion}
In this paper, we proved the convergence of empirical distribution defined by a Gaussian process by building finite sample bounds for multivariate Gaussian under general dependence. Loosely speaking, we can conclude that if randomness of the stochastic process originates from a Gaussian process whose correlation is not strong enough, the elements from the process will enjoy large number property. We think this conclusion will shed light on related fields regarding sequential randomness. \\

In addition, this paper demonstrated the framework to deal with empirical process under dependence structure. While previous work studied the dependence structure based on Markov property and martingale difference, it has been pointed out the multivariate Gaussian structure is sufficient to guarantee the convergence of empirical distribution. More broadly, while the property of multivariate Gaussian is interpreted by Hermite polynomials, we can deal with certain dependence for other distributions if there exists an appropriate Hilbert basis with respect to the bivariate dependence structure. What's more, if we adopt the framework of $L^2$ chaining and smoothing, we only have to study the bivariate dependence structure in each pair of random variables.

\section*{Acknowledgments}
The author would like to thank Jiantao Jiao at UC Berkeley for helpful discussions. The author would also like to thank Lihua Lei at Stanford University for reminding him of the result in \cite{Azriel2015The}, which implies the condition in Theorem \ref{thm1} is necessary and sufficient.\\

This paper serves as the undergraduate thesis of Jikai Hou at Peking University. Jikai Hou was supported by the elite undergraduate training program of School of Mathematical Sciences in Peking University.

\bibliographystyle{unsrt}
\bibliography{ref}

\end{document}